\documentclass[12pt]{amsart}
\usepackage{amssymb,amsfonts,amsmath,amsopn,amstext,amscd,color,
graphicx,latexsym,mathabx,mathrsfs,skak,stackrel,todonotes,xy,url,
verbatim,tikz-cd,enumitem}

\usepackage{hyperref}

\input xy
\xyoption{all}

\theoremstyle{remark}

\newtheorem{para}{\bf}[subsection]
\newtheorem{example}[para]{\bf Example}
\newtheorem{examples}[para]{\bf Examples}

\newtheorem{rem}[para]{\bf Remark}
\theoremstyle{definition}

\newtheorem{dfn}[para]{Definition}

\theoremstyle{plain}

\newtheorem{thm}[para]{Theorem}

\newtheorem{lemma}[para]{Lemma}

\newtheorem{cor}[para]{Corollary}
\newtheorem{prop}[para]{Proposition}

\newenvironment{numequation}{\addtocounter{para}{1}
\begin{equation}}{\end{equation}}

\newcommand{\bbD}{{\mathbb D}}

\newcommand{\bbQ}{{\mathbb Q}}

\newcommand{\bbZ}{{\mathbb Z}}

\newcommand{\bB}{{\bf B}}

\newcommand{\bG}{{\bf G}}

\newcommand{\bT}{{\bf T}}

\newcommand{\frb}{{\mathfrak b}}

\newcommand{\frg}{{\mathfrak g}}
\newcommand{\frh}{{\mathfrak h}}

\newcommand{\frl}{{\mathfrak l}}
\renewcommand{\frm}{{\mathfrak m}}

\newcommand{\frs}{{\mathfrak s}}
\newcommand{\frt}{{\mathfrak t}}
\newcommand{\fru}{{\mathfrak u}}

\newcommand{\frz}{{\mathfrak z}}

\newcommand{\cA}{{\mathcal A}}
\newcommand{\cB}{{\mathcal B}}
\newcommand{\cC}{{\mathcal C}}
\newcommand{\cD}{{\mathcal D}}

\newcommand{\cF}{{\mathcal F}}

\newcommand{\cO}{{\mathcal O}}

\newcommand{\ovE}{{\overline E}}

\makeatletter
\newcommand*\bdot{{\mathpalette\bdot@{.9}}}
\newcommand*\bdot@[2]{\mathbin{\vcenter{\hbox{\scalebox{#2}{$\m@th#1\bullet$}}}}}
\makeatother

\newcommand{\alg}{{\rm alg}}

\newcommand{\cOa}{{\cO_\alg}}
\newcommand{\cOi}{{\cO^\infty}}
\newcommand{\cOia}{{\cO^\infty_\alg}}

\newcommand{\DbcOi}{{D^b(\cOi)}}
\newcommand{\DbcOia}{{D^b(\cOia)}}

\newcommand{\DGmod}{D(G){\mbox{-}\rm{mod}}}

\newcommand{\End}{{\rm End}}

\newcommand{\Ext}{{\rm Ext}}

\newcommand{\frsl}{\frs\frl}

\newcommand{\hmod}{\mbox{-}{\rm mod}} 
\newcommand{\Hom}{{\rm Hom}}
\newcommand{\hra}{\hookrightarrow}
\newcommand{\id}{{\rm id}}

\newcommand{\isom}{\stackrel{\simeq}{\lra}}

\newcommand{\Lie}{{\rm Lie}}

\newcommand{\lra}{\longrightarrow}
\newcommand{\midc}{{\; | \;}}

\renewcommand{\o}{{\circ}}
\newcommand{\ot}{\otimes}
\newcommand{\pd}{{\rm pd}}

\newcommand{\pr}{{\rm pr}}

\newcommand{\Qp}{{\bbQ_p}}
\newcommand{\ra}{\rightarrow}

\newcommand{\RHom}{{\rm RHom}}

\newcommand{\rmd}{{\rm d}}

\newcommand{\smt}{{\rm smt}}

\newcommand{\sub}{\subset}

\newcommand{\UbE}{{U(\frb_E)}}

\newcommand{\Ug}{{U(\frg)}}
\newcommand{\UgE}{{U(\frg_E)}}
\newcommand{\Ugmodfg}{U(\frg_E){\mbox{-}\rm{mod}_{\rm fg}}}
\newcommand{\Ugmod}{U(\frg_E){\mbox{-}\rm{mod}}}
\newcommand{\Uh}{{U(\frh)}}

\newcommand{\UuoppE}{{U(\fru-_E)}}

\newcommand{\vcF}{{\check{\cF}}}
\newcommand{\vcFGB}{{\check{\cF}^G_B}}

\newcommand{\xra}{\xrightarrow}

\newcommand{\Z}{{\mathbb Z}}

\begin{document}

\title{Duality in derived category $\cO^\infty$}
\author{Cemile Kurkoglu}
\address{Denison University, OH, U.S.A.}
\email{kurkogluc@denison.edu}

\begin{abstract} Let $\bG$ be a split connected reductive group over a finite extension $F$ of $\Qp$, and let $\bT \sub \bB\sub \bG$ be a maximal split torus and a Borel subgroup, respectively. Denote by $G = \bG(F)$ and $B = \bB(F)$ their groups of $F$-valued points and by $\frg = \Lie(G)$ and $\frb = \Lie(B)$ their Lie algebras. Let $\cOi$ be the thick category $\cO$ for $(\frg,\frb)$, and denote by $\cOia \sub \cOi$ the full subcategory consisting of objects whose weights are in $X^*(\bT)$. Both are Serre subcategories of the category of all $U$-modules, where $U = \Ug$. We show first that the functor $\bbD^\frg = \RHom_U(-,U)$ preserves $D^b(U)_\cOia$, and we deduce from a result of Coulembier-Mazorchuk that the latter category is equivalent to $\DbcOia$. 
\end{abstract}

\maketitle

\tableofcontents

\section{Introduction}

\subsection{Overview of results and contents}

Let $F/\Qp$ be a finite extension of the field $\Qp$ of $p$-adic numbers, and let $\bG$ be a connected split reductive group over $F$. We fix a Borel subgroup $\bB \sub \bG$ and a maximal split torus $\bT \sub \bB$. Their Lie algebras we denote by gothic letters $\frg$, $\frb$, and $\frt$, respectively, and their groups of $F$-rational points by $G$, $B$, and $T$, respectively. 

\vskip8pt

First we introduce the Bernstein-Gelfand-Gelfand category $\cO$ for the pair $(\frg,\frb)$ (suitably defined for Lie algebras over a field which is not algebraically closed), its extension-closure $\cOi$, sometimes called the ``thick'' category $\cO$, and the subcategories $\cOa \sub \cO$ and $\cOia \sub \cOi$ consisting of modules $M$ whose weights are {\it algebraic} for the torus $\bT$, i.e., lie in the image of the canonical map $X^*(\bT) \xrightarrow{\rmd} \frt^*$. It is shown that $\cOi$ and $\cOia$ are Serre subcategories of the category of all $\Ug$-modules. 

\vskip5pt

Then we study duality functors on derived categories related to the previous categories $\cOi$ and $\cOia$. To discuss those, denote by $D^b_\cOi(\Ug\hmod)$ (resp. $D^b_\cOia(\Ug\hmod)$) the subcategory of the bounded derived category $D^b(\Ug\hmod)$ consisting of complexes whose cohomology modules lie in $\cOi$ (resp. $\cOia$). It is shown that these categories are invariant under the duality functor $M \rightsquigarrow \RHom_\Ug(M,\Ug)$. Furthermore, a theorem by K. Coulembier and V. Mazorchuk says that the canonical functors 
\[D^b(\cOi) \lra  D^b_\cOi(\Ug\hmod)\;,\hskip15pt D^b(\cOia) \ra  D^b_\cOia(\Ug\hmod)\]

\vskip5pt

are equivalences of categories. We thus obtain involutive functors
\[\bbD^\frg: \DbcOia \lra \DbcOia \;, \hskip15pt \bbD^\frg_\alg: \DbcOia \lra \DbcOia \;.\]

\vskip5pt

\vskip5pt

\subsection{Notations and conventions} By $F/\Qp$ we denote a finite extension, occasionally called the ``base field'', and by $E/F$ another finite extension, called the ``coefficient field''. The subscript ``$E$'' denotes the the base change to $E$, i.e., $W_E = W \ot_F E$ for some $F$-vector space $W$.

\vskip5pt

We let $\bG$ denote a connected split reductive group over $F$, by $\bB \sub \bG$ a Borel subgroup and by $\bT \sub \bB$ a maximal $F$-split torus, and write $\frg = \Lie(G) \supset \frb  = \Lie(\bB) \supset \frt = \Lie(\bT)$ for their Lie algebras. 

\vskip5pt

We denote by $\Phi = \Phi(\frg,\frt)$ the set of roots of $\frt$ on $\frg$, and by $\frg_\alpha$ the root space corresponding to $\alpha \in \Phi$. As usual, $\Phi^+$ denotes the set of roots determined by $\frb$, i.e., those $\alpha$ for which $\frg_\alpha$ is contained in $\frb$. We set $\frt^*_E = \Hom_F(\frt,E)$, and we let $\Gamma = \sum_{\alpha \in \Phi^+} \Z_{\ge 0} \alpha$ be the integral cone generated by the positive roots. $\frt^*_E$ carries a partial ordering $\prec$ defined by 
\[\lambda \prec \mu \;\;\Longleftrightarrow \;\; \mu -\lambda \in \Gamma \;.\]

\vskip8pt

Set $\rho = \frac{1}{2}\sum_{\alpha \in \Phi^+} \alpha$ and define the dot-action of the Weyl group $W$ of $\Phi$ by $w \cdot \lambda := w(\lambda + \rho)-\rho$. The $W$-orbit of $\lambda \in \frt^*_E$  for the dot-action is denoted by $|\lambda|$.

\vskip5pt

If $A$ is a ring (always associative and with unit element) then $A\hmod$ denotes the category of all $A$-left-modules. Given a (cochain) complex 
\[K^{\bdot} = (\ldots \xra{d_K^{i-1}} K^i \xra{d_K^i} K^{i+1} \xra{d_K^{i+1}} \ldots)\]

\vskip5pt

and $n \in \Z$ we let $K[n]$ be the complex $K[n]^i := K^{i+n}$ and with differential $d_{K[n]}^i = (-1)^n d_K^{i+n}$.

\vskip5pt

\section{Categories \texorpdfstring{$\cO$}{} and \texorpdfstring{$\cOi$}{}}

\subsection{Categories \texorpdfstring{$\cO$}{} and \texorpdfstring{$\cOa$}{}}
The material in this section has appeared before in \cite{OrlikStrauchJH}.

\begin{para}{\it Diagonalizable, $E$-split, and algebraic modules.}  We recall some concepts about representations of $\frt_E$. Given an $\frt_E$-module $(\phi,M)$ and a weight $\lambda \in \frt_E^*$, we set
\[M^i_\lambda = \{m \in M \midc \forall h \in \frt_E: (\phi(h)-\lambda(h)\cdot \id)^i.m = 0\}\]

\vskip5pt

and $M^\infty_\lambda = \bigcup_{i \ge 1} M^i_\lambda$. For $M^1_\lambda$ we write $M_\lambda$. 
Then $M_\lambda$ is called the eigenspace of $M$ for the weight $\lambda$, and $M^\infty_\lambda$ is called the generalized eigenspace of $M$ for the weight $\lambda$. 
\end{para}

\vskip5pt

\begin{dfn}\label{dfn split} Let $M$ be a $\frt_E$-module. 

\vskip5pt

(i) $M$ is called {\it diagonalizable} (over $E$) if 
\[M = \bigoplus_{\lambda \in \frt_E^*} M_\lambda \;.\] 

\vskip5pt

(ii) $M$ is called {\it $E$-split} if 
\[M = \bigoplus_{\lambda \in \frt^*_E} M^\infty_\lambda \;.\]

\vskip5pt 

If $M$ is $E$-split we set $\Pi(M) = \{\lambda \in \frt^*_E \midc M^\infty_\lambda \neq 0\}$ and call it the set of {\it weights} of $M$. 
\end{dfn}

\vskip5pt

For later use we note the following elementary 

\begin{lemma}\label{non-zero eigenspace} Given a $\frt_E$-module $M$ such that $M^\infty_\lambda \neq 0$, then $M^1_\lambda \neq 0$.
\end{lemma}

\begin{proof} Let $v \in M^\infty_\lambda$ be non-zero. Let $h_1, \ldots, h_\ell$ be a basis of $\frt$. Choose $i_1>0$ maximal such that $v_1 := (\phi_M(h_1)-\lambda(h_1)\cdot\id)^{i_1-1}.v \neq 0$. Note that $v_1 \in \ker(\phi_M(h_1)-\lambda(h_1)\cdot\id)$. Then choose $i_2>0$ maximal with the property that $v_2 := (\phi_M(h_2)-\lambda(h_2)\cdot\id)^{i_2-1}.v_1 \neq 0$. Note that $v_2 \in \ker(\phi_M(h_1)-\lambda(h_1)\cdot\id) \cap  \ker(\phi_M(h_2)-\lambda(h_2)\cdot\id)$. We keep repeating this construction until we have found a non-zero $w := w_h \in M^1_\lambda$.
\end{proof}

\vskip5pt

The group of algebraic characters $X^*(\bT) = \Hom_{{\rm alg.gps}/F}(\bT,\bG_{m,F})$ embeds into $\frt^*$ by the derivative $\rmd: X^*(\bT) \ra \frt^*$. We denote by $\frt^*_\alg$ the image of $\rmd$, and call the weights in $\frt^*_\alg$ {\it algebraic} with respect to $\bT$. A $\frt_E$-module $M$ is called {\it algebraic} with respect to $\bT$ if it is $E$-split and $M = \bigoplus_{\lambda \in \frt^*_\alg} M^\infty_\lambda$. 

\vskip5pt

As $\bT$ is fixed throughout this thesis, we often drop the qualifier ``with respect to $\bT$''.

\vskip5pt

\begin{dfn}\label{dfn O}
(i) Category $\cO$ for $(\frg, \frb, \frt)$ and the coefficient field $E$ is the full subcategory of $\Ugmod$ consisting of modules $M$ which satisfy the
following properties:
\begin{enumerate}
\item $M$ is finitely generated as a $\UgE$-module.
\item $M$ is diagonalizable. 
\item  The action of $\frb_E$ on $M$ is locally finite, i.e., for every $m \in M$, the subspace $\UbE.m \subset M$ is finite-dimensional over $E$.
\end{enumerate}

\vskip5pt

(ii) Category $\cOa$ for $(\frg, \frb, \bT)$ and the coefficient field $E$ is the full subcategory of $\cO$ consisting of modules $M$ which are algebraic with respect to $\bT$.  
\end{dfn}

\vskip5pt

\begin{para}\label{properties cO}{\it Properties of category $\cO$.}  Category $\cO$, as defined here for a split reductive Lie algebra over a field which is not algebraically closed, enjoys all the properties that category $\cO$ for a semisimple complex Lie algebra has. We only mention some of them here \cite[1.1, 1.11]{HumphreysBGG}:

\vskip5pt

\begin{enumerate}
\item Every $M$ in $\cO$ is noetherian and artinian (i.e., ascending chains of submodules are stationary, and descending chains of submodules are stationary). In particular every object in $\cO$ has finite length.
\item  $\cO$ is closed under submodules, quotients, and finite direct sums.
\item $\cO$ is an abelian category.
\item If $M$ in $\cO$ and $L$ is finite dimensional, then $L \otimes M$ also lies in $\cO$.
Thus $M \rightsquigarrow L \otimes M$ defines an exact functor $\cO \ra \cO$.
\item If $M$ in $\cO$, then $M$ is $\frz_E$-finite: for each $v \in M$, the span of $\{z \cdot v \,|\, z \in \frz_E \}$ is finite dimensional.
\item If $M$ in $\cO$, then $M$ is finitely generated as a $U(\fru-_E)$-module.

\end{enumerate}
\end{para}

\vskip5pt

\begin{para}{\it Verma modules and simple modules.}  For $\lambda \in \frt_E$ we denote by $E_\lambda$ the one-dmensional module given by $\lambda: \frt_E \ra \End_E(E) = E$. Via the quotient morphism $\frb_E \ra \frt_E$ we consider $E_\lambda$ also as a $\frb_E$-module, and thus as a $\UbE$-module. The $\UgE$-module 

\[M(\lambda) = \UgE \otimes_{\UbE} E_\lambda\]

\vskip5pt

is called the {\it Verma module} with highest weight $\lambda$. The vector $v_0 := 1 \otimes 1$ generates $M(\lambda)_\lambda$ as $E$-vector space and $v_0$ is a maximal vector in the sense that $\fru.v_0 = \{0\}$. The Verma module $M(\lambda)$ has following universal property: for any $N$ in $\Ugmod$ and any maximal vector $v \in N_\lambda$  there is a unique morphism $f: M(\lambda) \ra N$ in $\Ugmod$ such that $f(v_0) = v$. The module $M(\lambda)$ has a unique maximal submodule and therefore a unique simple quotient which we denote by $L(\lambda)$. Every simple object in $\cO$ is isomorphic to $L(\lambda)$ for a unique $\lambda$.
\end{para}

\vskip5pt

\begin{para}\label{O not closed under ext}{\it Category $\cO$ is not closed under extensions.}  For the very purpose of our paper, we would like to consider the subcategory $D^b_\cO(\Ugmod)$ of $D^b(\UgE)$ consisting of complexes $M^{\bdot} = (M^n)_n$ whose cohomology groups are in category $\cO$. If $\cA$ is an abelian category and $\cB$ is a weak Serre subcategory, $D^\star_\cB(\cA)$ is a strictly full\footnote{A subcategory $\cC$ of a category $\cA$ is called {\it strictly full} if it is a full subcategory and contains all objects of $\cA$ which are isomorphic to objects of $\cC$ \cite[4.2.10]{stacks-project}.} saturated\footnote{Let $\cD$ be a pre-triangulated category. A full pre-triangulated subcategory $\cD'$ of $\cD$ is {\it saturated} if whenever $X \oplus Y$ is isomorphic to an object of $\cD'$ then both $X$ and $Y$ are isomorphic to objects of $\cD'$ \cite[13.6.1]{stacks-project}.} triangulated of $D^\star(\cA)$, where $\star \in \{\emptyset,-,+,b\}$ \cite[13.17.1]{stacks-project}. But as the following example shows, $\cOa$, and hence $\cO$, is not stable under extensions in $\Ugmod$.
\end{para}

\vskip5pt

\begin{example}\label{exampleextension}
 Let $\frg = \frsl(2,F)$, and identify $\lambda \in  \frh_{\ast}$ with a scalar. Let $N = Fe_1 \oplus Fe_2$ be a 2-dimensional $U(\frb)$-module defined by letting $x$ act as $0$ and $h$ act as $h.e_1 = \lambda e_1$ and $h.e_2 = e_1 + \lambda e_2$. The induced $U(\frg)$-module $M := U(\frg) \otimes_{U(\frb)} N$ fits
into an exact sequence which fails to split:
\begin{equation}
    0 \ra M(\lambda) \ra M \ra M(\lambda) \ra 0
\end{equation}
\end{example}

\vskip5pt

\begin{para}\label{Serre}{\it Serre subcategories.}  Let $\cA$ be an abelian category. Recall that a subcategory $\cC$ of an abelian category $\cA$ is called a {\it Serre subcategory} (resp. {\it weak Serre subcategory}) if it is non-empty, full, and if for any exact sequence 
\[M_0 \lra M_1 \lra M_2 \lra M_3 \lra M_4\] 

\vskip5pt

in $\cA$ the object $M_2$ belongs to $\cC$ if both $M_1$ and $M_3$ (resp. if all of $M_0, M_1, M_3, M_4$) belong to $\cC$ \cite[12.10.1]{stacks-project}, \cite[Ex. 10.3.2]{WeibelH}. A subcategory $\cC$ of $\cA$ is a Serre subcategory if and only if (i) it contains the zero object, (ii) is strictly full, and (iii) is stable under subobjects, quotients, and extensions \cite[12.10.2]{stacks-project}. 

\vskip5pt

From the example \ref{exampleextension} above, it follows that $\cO$ and $\cOa$, are {\it not} weak Serre subcategories of $\Ugmod$.
\end{para}

\vskip5pt

\subsection{Categories \texorpdfstring{$\cOi$}{} and \texorpdfstring{$\cOia$}{}}
As category $\cO$ is not stable under extensions, we are led to consider the smallest strictly full abelian subcategory of $\Ugmod$ which contains $\cO$ (resp. $\cOa$) and is closed under extensions. We see from the example \ref{exampleextension} that, in general, $\frt_E$ acts no longer acts diagonalizable on modules in this subcategory.

\vskip5pt

\begin{dfn}\label{dfn Oinfty} 
(i) Category $\cOi$ for $(\frg, \frb, \frt)$ and the coefficient field $E$ is the full subcategory of $\Ugmod$ consisting of modules $M$ which satisfy the
following properties:
\begin{enumerate}
\item $M$ is finitely generated as a $\UgE$-module.
\item $M$ is $E$-split as a $\frt_E$-module. 
\item The action of $\frb_E$ on $M$ is locally finite, i.e. for every $m \in M$, the subspace $\UbE.m \subset M$ is finite-dimensional over $E$.
\end{enumerate}

\vskip5pt

(ii) Category $\cOia$ for $(\frg, \frb, \bT)$ and the coefficient field $E$ is the full subcategory of $\cOi$ consisting of modules $M$ which are algebraic with respect to $\bT$.  
\end{dfn}

\vskip5pt

\begin{rem}\label{references to Oinfty}  Category $\cOi$ is sometimes called ``thick category $\cO$''. It has been studied in several papers, for example  
\cite{Soergel_Diplomarbeit} and \cite{CoulembierMazorchuk_III}.
\end{rem}

\vskip5pt

\begin{prop}\label{Oinfty abelian} Categories $\cOi$ and $\cOia$ are both abelian categories. Every object $M$ in $\cOi$ is noetherian (i.e., any ascending chain of submodules is stationary).  
\end{prop}

\begin{proof} As they are defined as full subcategories of the abelian category $\Ugmod$, one only needs to check that these categories are closed under taking finite direct sums, kernels, and cokernels. This is straightforward to verify using the fact that $\UgE$ is noetherian (hence any submodule of a finitely generated module is finitely generated). This also implies the assertion about modules in $\cOi$ being noetherian.  
\end{proof}

\begin{lemma}\label{lemma Oinfty} Let $M$ be in $\cOi$.

\vskip5pt

(i) For $\alpha \in \Phi$ and $i \in \bbZ_{\ge 1}$ one has $x_\alpha.M^i_\lambda \sub M^i_{\lambda + \alpha}$. In particular, $x_\alpha.M^\infty_\lambda \sub M^\infty_{\lambda + \alpha}$. 

\vskip5pt

(ii) $\Pi(M)$ is contained in a finite union of sets of the form $\lambda - \Gamma$, where $\Gamma = \sum_{\alpha \in \Phi^+} \bbZ_{\ge 0} \alpha$. 

\vskip5pt

(iii) The generalized eigenspace $M_\lambda^\infty$ is finite-dimensional over $E$ for any weight $\lambda \in \Pi(M)$

\vskip5pt

(iv) The subspace $M^i := \bigoplus_{\lambda \in \Pi(M)} M^i_\lambda$ is a $\UgE$-submodule of $M$ which belongs to $\cOi$. Moreover, $M^1$ is a submodule which lies in $\cO$ and which is non-zero if $M$ is non-zero. There is $i \in \bbZ_{>0}$ such that $M = M^i$. 
\end{lemma}

\begin{proof} 
(i) For $h \in \frt_E$ one has $hx_\alpha = \alpha(h)x_\alpha + x_\alpha h$, and hence 
\[(h - (\lambda(h)+\alpha(h))\cdot 1)x_\alpha = x_\alpha(h - \lambda(h)\cdot 1) \;,\]

\vskip5pt

and thus $(h - (\lambda(h)+\alpha(h))\cdot 1)^i x_\alpha = x_\alpha(h - \lambda(h)\cdot 1)^i$. This shows that 
\[x_\alpha .\ker\Big((\phi_M(h) - \lambda(h)\cdot \id)^i\Big) \sub \ker\Big((\phi_M(h) - (\lambda+\alpha)(h)\cdot \id)^i\Big) \;.\]

\vskip5pt

(ii) Let $M$ be generated by the elements $m_1, \ldots, m_n$. Since $\UbE.m_j$ is finite-dimensional 

\begin{enumerate}
    \item there is $i \gg 0$, and 
    \item there are finitely many $\lambda_1, \ldots,\lambda_r$ in $\frt^*_E$, and
    \item there are finite-dimensional subspaces $W_k \sub M^i_{\lambda_k}$, for $k = 1,\ldots,n$
\end{enumerate} 

\vskip5pt

such that one has $\UbE.m_j \sub \bigoplus_{k = 1}^r W_k$ for every $j \in \{1, \ldots,n\}$. By the PBW theorem we have, as $E$-vectors space, $\UgE = \UuoppE \ot_E \UbE$. By (i) we find that 
\begin{numequation}\label{weights estimate}
M = \sum_{j=1}^n \UgE.m_j \sub \UuoppE.\bigoplus_{k = 1}^r W_k \sub \sum_{k=1}^r \sum_{\nu \in \lambda_k - \Gamma} (M^i_\nu \cap \UuoppE.W_k)\;.
\end{numequation}
(iii) This follows from \ref{weights estimate} together with the fact that for any $\mu \in \Lambda_r$ there are only finitely many $(m_1, \ldots,m_t) \in \bbZ_{\ge 0}^t$ such that $\mu  = \sum_{j=1}^t m_j \beta_j$, where $\{\beta_1, \ldots,\beta_t \} = \Phi^+$ and $t =|\Phi^+|$. 

\vskip5pt

(iv) That $M^i$ is a $\UgE$-submodule follows from (i). By \ref{Oinfty abelian}, it lies in category $\cOi$. That $M^1$ lies in $\cO$ follows by definition. That $M^1$ is non-zero if $M$ is non-zero follows from \ref{non-zero eigenspace}. The last assertion is a consequence of \ref{weights estimate}. Alternatively, one could argue that the sequence of submodules $M^1 \sub M^2 \sub M^3 \ldots$ must be stationary as $M$ is noetherian, by \ref{Oinfty abelian}.
\end{proof}

\vskip5pt

\begin{prop}\label{key facts Oinfty} (i) Categories $\cOi$ and $\cOia$ are closed under extensions and are Serre subcategories of $\Ugmod$.

\vskip5pt

(ii) A $\UgE$-module $M$ belongs to $\cOi$ if and only if there is a finite filtration $0 = M_0 \sub M_1 \sub \ldots \sub M_n = M$ such that all quotients $M_k/M_{k-1}$, $1 \le k \le n$, are in category $\cO$.

\vskip5pt

(iii) A $\UgE$-module $M$ belongs to $\cOia$ if and only if there is a finite filtration $0 = M_0 \sub M_1 \sub \ldots \sub M_n = M$ such that all quotients $M_k/M_{k-1}$, $1 \le k \le n$, are in category $\cOa$.

\vskip5pt

(iv) Simple objects in category $\cOi$ (resp. $\cOia$) lie in $\cO$ (resp. $\cOa$), and every object in $\cOi$ (resp. $\cOia$) has finite length. 

\end{prop}

\begin{proof} (i) Suppose first that $N, L$ belong to category $\cOi$, and that $0 \ra N \hra M \xra{\pi} L \ra 0$ is an exact sequence in $\Ugmod$. It is clear that $M$ is then finitely generated. Let $W \sub M$ be a finite-dimensional $\frt_E$-stable subspace. We want to show that $W$ is $E$-split in the sense of \ref{dfn split}. Let $\ovE$ be the algebraic closure of $E$. We use the fact that the induced action of $\frt_\ovE$ on $W_\ovE = W \ot_E \ovE$ is $\ovE$-split, cf. \cite[II, sec. 4, Thm. 7]{JacobsonLie}. So let $W_\ovE = W_{\ovE,\lambda_1}^\infty \oplus \ldots \oplus W_{\ovE,\lambda_r}^\infty$ be a decomposition of $W_\ovE$ into simultaneous generalized eigenspaces, for pairwise different weights $\lambda_i: \frt \ra \ovE$. Note that the eigenvalues of $h \in \frt_E$ on $W_\ovE$ are thus $\lambda_1(h),\ldots,\lambda_r(h)$ (not necessarily pairwise distinct). Then the eigenvalues of $h$ on $W_\ovE \cap N_\ovE$ and the eigenvalues of $h$ on $\pi(W_\ovE)$ are also among $\lambda_1(h),\ldots,\lambda_r(h)$. But the eigenvalues of $h$ on any finite-dimensional $h$-stable subspace of $N_\ovE$ (resp. $L_\ovE$) are in $E$, as $N$ (resp. $L$) is $E$-split. The linear forms $\lambda_1,\ldots,\lambda_r$ are therefore $E$-valued and thus belong to $\frt_E$. As $M$ is the union of its finite-dimensional $\frt_E$-stable subspaces, $M$ is $E$-split.  

\vskip5pt

We also note that for $\lambda \in \frt^*_E$ we have an exact sequence $0 \ra N_\lambda^\infty \ra M_\lambda^\infty \ra L_\lambda^\infty \ra 0$. By \ref{lemma Oinfty}, it follows that $\Pi(M) = \Pi(N) \cup \Pi(L)$ is contained in a finite union of sets of the form $\lambda - \Gamma$. Therefore, for any $\nu \in \Pi(M)$ we have that $(\nu + \Gamma) \cap \Pi(M)$ is finite. Now we have 

\[\UbE.M_\nu^\infty \sub \sum_{\nu \in (\nu + \Gamma) \cap \Pi(M)} M^\infty_\nu \;.\] 

\vskip5pt

By \ref{lemma Oinfty}, the generalized eigenspaces are finite-dimensional, which implies that $\UbE.M^\nu_\lambda$ is finite-dimensional. This completes the proof that $\cOi$ is closed under extensions. That the same is true for $\cOia$ follows from the fact that $\Pi(M) = \Pi(N) \cup \Pi(L)$ when $M$ is an extension of $L$ by $N$. 

\vskip5pt

It follows from the remark in \ref{Serre} that both $\cOi$ and $\cOia$ are Serre subcategories of $\Ugmod$.

\vskip5pt

(ii) and (iii) ``$\Longleftarrow$'' For any $i \in \{1,\ldots,n\}$ there is an exact sequence 
\[0 \ra M_{i-1} \ra M_i \ra M_i/M_{i-1} \ra 0 \;.\] 

\vskip5pt

Using part (i), it follows by induction on $i$ and our assumption that $M_i$ is in category $\cOi$ (resp. $\cOia$). Hence $M = M_n$ belongs to category $\cOi$ (resp. $\cOia$).  

\vskip5pt

``$\Longrightarrow$'' Given $M$, we know from \ref{lemma Oinfty} that there is $i \gg 0$ such that $M = M^i$. Furthermore,
\[0 \sub M^1 \sub M^2 \sub \ldots \sub M^i = M\]

\vskip5pt

is a filtration of $M$ by submodules, and $M^i/M^{i-1}$ lies in category $\cO$ (resp. $\cOa$).

\vskip5pt

(iv) Let $L$ be a simple module in $\cOi$. By \ref{lemma Oinfty}, $L^1 = \bigoplus_{\lambda \in \Pi(L)} L^1_\lambda \sub L$ is a non-zero submodule, and hence must be equal to $L$. By (iii) that every object in $\cOi$ and $\cOia$ possesses a finite filtration whose successive subquotients are in category $\cO$. As objects in category $\cO$ have finite length, the same holds for objects in $\cOi$ and $\cOia$.
\end{proof}

\subsection{Blocks and \texorpdfstring{$\frz$}{}-blocks}

\begin{para}\label{z-blocks}{\it $\frz$-blocks\footnote{We use the term ``$\frz$-block'' here as we have not found a generally accepted name for those subcategories. In \cite[end of 1.13]{HumphreysBGG} it is said that those subcategories are sometimes also referred to as ``blocks'', but in order to avoid confusion with blocks in the sense of the theory of abelian categories, we prefer not to call them blocks here.} of category $\cOi$.} 
 As we have recalled in \ref{properties cO}, for every $M$ in $\cO$ every element $m \in M$ is $\frz_E$-finite, i.e., annihilated by an ideal $J \sub \frz_E$ of finite codimension (i.e., $\dim_E(\frz_E/J) < \infty$). As $M$ is finitely generated, it follows that $M$ is annihilated by an ideal of finite codimension, $I \sub \frz_E$ say. By \ref{key facts Oinfty}, this also holds for objects in $\cOi$. 

\vskip5pt

Because $\frz_E/I$ is an Artin ring, we can apply the decomposition of Artin rings into it's local components: $\frz_E/I = \prod_{i=1}^r \frz_E/I_i$, where $\frm_i := \sqrt{I_i}$ is a maximal ideal. Then $\frz_E/I_i = (\frz_E/I)_{\frm_i}$ is the localization of $\frz_E/I$ at $\frm_i$. It follows that 
\[M \; = \; (\frz_E/I) \ot_{\frz_E} M \; = \; \bigoplus_{i=1}^r \; (\frz_E/I_i) \ot_{\frz_E} M \; = \; \bigoplus_{i=1}^r \; (\frz_E/I)_{\frm_i} \ot_{\frz_E} M \; = \; \bigoplus_{i=1}^r M_{\frm_i} \;,\]

\vskip5pt

and every $M_{\frm_i}$ is a $\UgE$-submodule. Moreover, it follows from the condition that modules in $\cOi$ are $E$-split that if $\frm$ is a maximal ideal of $\frz_E$, then $M_\frm = 0$ unless $\frz_E/\frm  = E$ (i.e., the canonical map $E \ra \frz_E/\frm$ is an isomorphism). In that case the map $\frz_E \ra \frz_E/\frm = E$ is of the form $\chi_\lambda$, where $\lambda \in \frt^*_E$ and $\chi_\lambda$ denotes the character by which $\frz_E$ acts on the Verma module $M(\lambda)$. One has $\chi_\lambda = \chi_\mu$ if and only if $|\lambda| = |\mu|$, cf. \cite[4.115]{KnappVogan_Cohomological}. Set $M_{|\lambda|} := M_{\ker(\chi_\lambda)}$ and write $\pr_{|\lambda|}: M \ra M_{|\lambda|}$ for the corresponding projection. 

\vskip5pt

For $\lambda \in \frt^*_E$ let $\cOi_{|\lambda|}$ be the full subcategory of $\cOi$ consisting of modules $M$ such that $M = M_{|\lambda|}$. Then every $M$ in $\cOi$ splits as $M = \bigoplus_{|\lambda|} M_{|\lambda|}$. We call $\cOi_{|\lambda|}$ a {\it $\frz$-block}. In general, it is not a block of this category. We obtain a decomposition $\cOi = \bigoplus_{|\lambda|} \cOi_{|\lambda|}$. Similarly, we have a decomposition
\[\cOia \; = \; \bigoplus_{|\lambda|,\,\lambda \, {\rm algebraic}} \cOi_{\alg,|\lambda|} \;.\]

\vskip5pt

We also define $\cO_{|\lambda|}$ to be the subcategory of modules $M$ such that $M = M_{|\lambda|}$, and call it a $\frz$-block of $\cO$.
\end{para}

\vskip5pt

\begin{para}\label{blocks}{\it Blocks of category $\cO$.}  If $\lambda$ is integral, then $\cO_{|\lambda|}$ is a block of category $\cO$ in the sense of the theory of abelian categories, cf. \cite[1.13]{HumphreysBGG}. For general $\lambda$ the subcategory $\cO_{|\lambda|}$ splits further into subcategories $\cO_\nu$ as follows. Given $\nu \in |\lambda|$ set $W_{[\nu]} = \{ w \in W \midc w \cdot \nu -\nu \in \Lambda_r\}$ and consider the subcategory of modules $M$ in $\cO$ which have the property that all irreducible subquotients are of the form $L(w \cdot \nu)$ with $w \in W_{[\nu]}$. Then $\cO_\nu$ is a block of $\cO$ and 
\[\cO_{|\lambda|}  = \bigoplus_{W/W_{[\lambda]}}\cO_{w \cdot \lambda} \;,\] 

\vskip5pt
cf. \cite[4.9]{HumphreysBGG}. For category $\cOi$ one can define in an analogous way $\cOi_\nu$ as the subcategory of modules $M$ in $\cOi$ which have the property that all irreducible subquotients are of the form $L(w \cdot \nu)$ with $w \in W_{[\nu]}$. While we have not tried to prove it, and it will not be relevant for our purposes later on, it seems natural to guess that these are the blocks of $\cOi$.
\end{para}

\section{Duality in the derived category of \texorpdfstring{$\cOi$}{}}

In this chapter we study the modules $\Ext^i_{\UgE}(M,\UgE)$ for modules $M$ in category $\cOi$ and $\cOia$. Our first result will be that these modules are again in $\cOi$ and $\cOia$, respectively. We also consider the derived functor $\RHom_{\UgE}(-,\UgE)$ on the derived category $D^b(\UgE)$ and the subcategory $D^b_{\cOi}(\UgE)$ of complexes of modules whose cohomology modules lie in $\cOi$.   

\vskip5pt

\subsection{Ext-duals of induced modules}

\begin{para}\label{right-to-left} {\it Preliminaries on left-modules and right-modules.}  To simplify notation we will write from now on 
\[\Ext^i_U(M,U) \; \mbox{ instead of } \; \Ext^i_{\UgE}(M,\UgE)\] 

and 
\[\RHom_U(-,U) \; \mbox{ instead of } \; \RHom_{\UgE}(-,\UgE) \;.\] 

\vskip5pt

If $M$ is a left-module for $\UgE$, then the dual space $M' = \Hom_E(M,E)$ is naturally a {\it right}-$\UgE$-module via
\[(f.x)(m) = f(x.m) \;\]

\vskip5pt

for all $f \in M'$, $x \in \frg_E$, and $m \in M$. Using the anti-isomorphism $\iota: \UgE \ra \UgE$ we can consider any $\UgE$-right-module $N$ as a $\UgE$-left-module by setting 
\[u.n := n.\iota(u) \;,\]

\vskip5pt

for all $u \in \UgE$ and $n \in N$. We write ${}_\ell N$ for this $\UgE$-left-module. If there is a chance that confusion may arise regarding left- or right-module structures, we will clarify which is meant. If $\frh \sub \frg_E$ is a subalgebra and $\lambda: \frh \ra E$ a linear form, then we write $E_\lambda$ for the one-dmensional $\frh$-left-module given by $x.1 = \lambda(x)$. We identify the dual space $E_\lambda'$ with $E$ by the map $E_\lambda' \ra E$, $f \mapsto f(1)$. If we do so, then $1.x = \lambda(x)$ is the formula for the natural $\frh$-{\it right}-module structure. In this case we also write $E_\lambda$ for this $\frh$-right-module.

\vskip5pt

If not indicated otherwise, then we always consider $\frh$ as a $\frh$-left-module via the adjoint action, i.e., $x.y := [x,y]$ for all $x,y \in \frh$. The dual space $\frh'$ is then naturally a $\frh$-right-module, and so is any exterior power $\bigwedge^i \frh'$.
\end{para}

\vskip5pt

\begin{thm}\label{ExtofInd} Let $\frh \sub \frg_E$ be a subalgebra, and let $W$ be a finite-dimensional $\frh$-module. Then there is an isomorphism in $D^b(\UgE^\o)$

\[\RHom_U(U \ot_{\Uh} W,U) \cong \Big(W' \ot_E \bigwedge^{\dim_F(\frh)} \frh'\Big) \ot_{\Uh} \UgE [-\dim_F(\frh)]\;.\]

\vskip5pt

In particular, 
\[\Ext^i(U \ot_{\Uh} W,U) \cong \left\{\begin{array}{ccl} \Big(W' \ot_E \bigwedge^{\dim_F(\frh)} \frh'\Big) \ot_{\Uh} \UgE & , & i  = \dim_F(\frh) \\
&& \\
0 &  , & i \neq \dim_F(\frh) 
\end{array}\right.\]
\end{thm}

\begin{proof} For the second formula see \cite[p. 386]{Chemla_Duality_Lie_algebroids}. The first formula follows from the second formula because a complex $K^{\bdot}$ which has non-zero cohomology only in degree $n$ is quasi-isomorphic to the complex $H^n(K^{\bdot})[-n]$, where $H^n(K^{\bdot})$ is the complex which has $H^n(K^{\bdot})$ in degree zero and the zero module in all other degrees.
\end{proof}

\begin{cor}\label{ExtofVerma} Given $\lambda \in \frt^*_E$, let $M(\lambda) = \UgE \ot_{\UbE} E_\lambda$ be the Verma module with highest weight $\lambda$. Then 
\[\RHom_U(M(\lambda),U) \cong \Big(E_{\lambda + 2\rho} \ot_{\UbE} \UgE \Big)[-\dim_F(\frb)]\;.\]

\vskip5pt

If we consider this $\UgE$-right-module as a $\UgE$-left-module, as explained in \ref{right-to-left}, then we have
\begin{numequation}\label{ExtofVerma left}
{}_\ell \RHom_U(M(\lambda),U) \cong M(-\lambda-2\rho)[-\dim_F(\frb)]\;.
\end{numequation}
In particular, as $\UgE$-left-modules one has
\[\Ext^i(M(\lambda),U) \cong \left\{\begin{array}{ccl} M(-\lambda-2\rho) & , & i  = \dim_F(\frb) \\
&& \\
0 &  , & i \neq \dim_F(\frb)\,. 
\end{array}\right.\]
\end{cor}

\begin{proof} Since $\bigwedge^{\dim_F(\frb)} \frb$ is a one-dimensional representation, it factors through the map $\frb \ra \frb/[\frb,\frb] \cong \frt$. Furthermore, because $\frb = \frt \oplus \bigoplus_{\beta \in \Phi^+} \frg_\beta$ we have
\[\bigwedge^{\dim_F(\frb)} \frb \; \cong \; \Big(\bigwedge^{\dim_F(\frt)} \frt\Big)\ot_E \bigotimes_{\beta \in \Phi^+}\frg_\beta ;,\]

\vskip5pt

as $\frt$-modules. This shows that this representation is equal to $E_{2 \rho}$. It follows from this that the {\it right} $\frb$-module $\bigwedge^{\dim_F(\frb)} \frb'$ is naturally isomorphic to $E_{2\rho}$ too, cf. \ref{right-to-left}. By the same reference, we have $E_\lambda' = E_\lambda$ as a $\UbE$-right-module. Hence $E_\lambda' \otimes_E \bigwedge^{\dim_F(\frb)} \frb_E' \cong E_\lambda \ot_E E_{2\rho} = E_{\lambda + 2\rho}$. It is straightforward to check that the map 
\[\UgE \ot_{\UbE} E_{-\lambda-2\rho} \lra {}_\ell \Big(E_{\lambda + 2\rho} \ot_{\UbE} \UgE\Big)\;,\;\; u \ot c \mapsto c \ot \iota(u) \;,\]

\vskip5pt

for all $c \in E$ and $u \in \UgE$ is well-defined and an isomorphism of $\UgE$-left-modules.  
\end{proof}

\subsection{The functors \texorpdfstring{$\Ext^i_U(-,U)$}{} preserve \texorpdfstring{$\cOi$}{} and \texorpdfstring{$\cOia$}{}}

\begin{prop}\label{Oinftystability} (i) For all $M$ in $\cOi$ and all $i \ge 0$ the $\UgE$-module $\Ext_U^i(M,U)$ is in $\cOi$ too. 

\vskip5pt

(ii) For all $M$ in $\cOia$, and all $i \ge 0$ the $\UgE$-module $\Ext_U^i(M,U)$ is in $\cOia$ too.

\vskip5pt

(iii) $\Ext_U^i(M,U) =0$ for $i>\dim_F(\frg)$.
\end{prop}

\begin{proof} (i) We prove this assertion in several steps. 

\vskip5pt

{\it Step 1: For modules possessing a standard filtration.} Suppose $M$ possesses a so-called standard filtration 
\begin{numequation}\label{standard filtration}
0  = M_0 \subsetneq M_1 \subsetneq \ldots \subsetneq M_n =M \;,
\end{numequation}
 which means that every successive quotient $M_i/M_{i-1}$, $i = 1, \ldots, n$, is isomorphic to a Verma module. We show the assertion by induction on $n$. The case $n=1$ has already been dealt with in  \ref{ExtofVerma}. Now suppose $n \ge 2$ and that the assertion is true for all modules possessing a standard filtration of length at most $n-1$. Let $M$ have a standard filtration of length $n$ as in \ref{standard filtration}, and set $N = M_{n-1}$, $V = M/N$. Then $V$ is isomorphic to a Verma module and we consider the long exact cohomology sequence
\begin{numequation}\label{long Ext sequence}
\ldots \lra \Ext_U^i(V,U) \lra \Ext_U^i(M,U) \lra \Ext_U^i(N,U) \lra \ldots
\end{numequation} 
The induction hypothesis implies that the modules $\Ext_U^i(V,U)$ and $\Ext_U^i(N,U)$ are in $\cOi$. By \ref{key facts Oinfty}, the module $\Ext_U^i(M,U)$ is in $\cOi$ too. 

\vskip5pt

{\it Step 2: For modules in category $\cO$.} By \cite[3.10]{HumphreysBGG}, every projective module in $\cO$ has a standard filtration, hence the assertion is true for projective modules. By \cite[Thm. 6]{BGG_categoryofgmodules}, every module $M$ in $\cO$ has finite projective dimension $\pd(M)$, i.e., there is an exact sequence
\begin{numequation}\label{projective resolution}
0 \lra P_n \lra P_{n-1} \lra \ldots \lra P_0 \lra M \lra 0
\end{numequation}
with $n = \pd(M)$ and modules $P_i$ which are projective in $\cO$, and $n$ is minimal with this property. We now prove the assertion in (i) by induction on $\pd(M)$. If $\pd(M) =0$ then $M$ is projective and there is nothing to show. Suppose $n>0$ and the assertion is true for all modules of projective dimension at most $n-1$. Let $M$ be of projective dimension $n$ and consider a resolution by projective modules in $\cO$ as in \ref{projective resolution}. Let $N$ be the image of $P_1 \ra P_0$ so that we have an exact sequence $0 \ra N \ra P_0 \ra M \ra 0$. Note that $\pd(N) \le n-1$. The long exact cohomology sequence
\[\ldots \lra \Ext_U^{i-1}(N,U) \lra \Ext_U^i(M,U) \lra \Ext_U^i(P_0,U) \lra \ldots \;,\]

\vskip5pt

together with our induction hypothesis, implies then the assertion for $M$. 

\vskip5pt

{\it Step 3: For all modules in $\cOi$.} By \ref{key facts Oinfty}, every object in $\cO^\infty$ has finite length. We prove the assertion by induction on the length $\ell(M)$. If $\ell(M)=1$ then $M$ is simple and hence belongs to $\cO$, again by \ref{key facts Oinfty}, and the assertion is true for $M$. Let $n = \ell(M) > 1$ and assume the assertion is true for all modules of length at most $n-1$. Choose a proper maximal submodule $N \subsetneq M$ and consider the exact sequence $0 \ra N \ra M \ra L := M/N \ra 0$. Applying the long exact cohomology sequence for $\Ext^{\bdot}_U(-,U)$ to this sequence we conclude as in Step 1.

\vskip5pt

(ii) For category $\cOia$ we can prove the assertion by the same arguments as above. However, we need to check that all arguments used above apply within the setting of category $\cOia$. In Step 1 we have used the result \ref{ExtofVerma}. Note that if $\lambda$ is algebraic, then so is $-\lambda -2\rho$, since $2\rho \in \Lambda_r$. Hence $\Ext^i_U(M(\lambda),U)$ is in $\cOia$ if $M(\lambda)$ is in $\cOia$ (which is equivalent to $\lambda$ being algebraic). Furthermore, if $M$ belongs to $\cOia$ and possesses a standard filtration as in \ref{standard filtration}, then all submodules and all quotient modules (which are Verma modules) are in $\cOia$, and assertion (ii) is true for $M$. Now let $M$ be in $\cOia$ and consider a projective resolution of $M$ as in \ref{projective resolution}. Without loss of generality we may assume that $M$ belongs to a $\frz$-block $\cO^\infty_{|\lambda|}$ of $\cOi$, as defined in \ref{z-blocks}. This implies that $\lambda$ is algebraic. Then we apply the projection $\pr_{|\lambda|}$ to the sequence \ref{projective resolution}. Note that $\pr_{|\lambda|}$ is the same as the translation functor $T^\lambda_\lambda$ of \cite[7.1]{HumphreysBGG}, which is exact and maps projective modules in $\cO$ to projective modules \cite[7.1]{HumphreysBGG}. Therefore, we may assume without loss of generality that all projective modules in \ref{projective resolution} belong to $\cOi_{|\lambda|}$ and are hence in $\cOia$. This shows that the arguments in Step 2 also apply to $\cOia$. For Step 3 this is immediate, as $\cOia$ is an abelian category. 

\vskip5pt

(iii) This follows from the Chevalley-Eilenberg resulution, cf. \cite[7.7.4]{WeibelH}. 
\end{proof}

\subsection{Dualizing modules for universal enveloping algebras}

We collect here some information about dualizing complexes and dualizing modules from the papers \cite{YekutieliZhang,Yekutieli_Dualizing_Ug}. Then we consider in particular the case of such modules for universal enveloping algebras.  

\begin{para}\label{generalities on dualizing modules}{\it Generalities on dualizing complexes.}  Let $k$ be a a field and $A,B,C$ unital associative $k$-algebras. We write $B^\o$ for the opposite algebra. The category $B^\o\hmod$ of left $B^\o$-modules is thus the same as the category of right $B$-modules. Given objects $M$ in $D(A \ot_k B^\o\hmod)$ and $N$ in $D(A \ot_k C^\o\hmod)$ with $M$ bounded from above or $N$ bounded from below, there is an object
\[\RHom_A(M,N) \mbox{ of } D(B \ot_k C^\o\hmod) \;.\]

\vskip5pt

It is calculated by replacing $M$ by an isomorphic complex in $D^-(A \ot_k B^\o\hmod)$ which consists of projective modules over $A$, or by replacing $N$ by an isomorphic complex in $D^+(A \ot_k B^\o\hmod)$ which consists of injective modules over $A$. For modules $M,N$, viewed as complexes concentrated in degree zero, one has
\[H^q \RHom_A(M,N) = \Ext^q_A(M,N) \;.\]

\vskip5pt
A complex $N$ in $D^+(A\hmod)$ is said to have {\it finite injective dimension} if there is $q_0 \in \Z$ such that for all $M$ in $A\hmod$ one has $\Ext^q(M,N)$ for all $q \ge q_0$.
\end{para}

\vskip5pt

\begin{dfn}\label{dfn dualizing complex} {\rm \cite[1.1]{YekutieliZhang} Assume $A$ and $B$ are $k$-algebras, with $A$ left noetherian
and $B$ right noetherian. 

\vskip5pt

(i) A complex $R$ in $D^b(A \ot_k B^\o\hmod)$ is called a {\it dualizing complex} if it satisfies the following three conditions:
\begin{enumerate}
\item[(1)] $R$ has finite injective dimension over $A$ and $B^\o$.
\item[(2)] $R$ has finitely generated cohomology modules over $A$ and $B^\o$.
\item[(3)] The canonical morphisms 
\[\begin{array}{lcl}B \lra \RHom_A(R,R) & \; \mbox{ in } \; & D(B \ot_k B^\o\hmod)\\
A \lra \RHom_{B^\o}(R,R) & \; \mbox{ in } \; & D(A\ot_k A^\o\hmod)
\end{array}\] 

\vskip5pt

are both isomorphisms.
\end{enumerate}

\vskip5pt

(ii) Assume now $A=B$. A dualizing complex $R$ in $D^b(A \ot_k A^\o\hmod)$ is called {\it rigid} if there is an isomorphism
\[R \lra \RHom_{A \ot_k A^\o}(A,R \ot_k R)\]

in $D(A \ot_k A^\o\hmod)$. Such an isomorphism is called a {\it rigidifying isomorphism}.\qed}
\end{dfn}

\vskip5pt

In case $A=B$, we shall say that $R$ is a dualizing complex over $A$. If a dualizing complex exists and is isomorphic to a module, considered as a complex concentrated in degree zero, we call it a {\it dualizing module}.

\vskip5pt

\begin{rem}\label{shift remark}
If $R$ is a dualizing complex, then so is any shift $R[n]$. However, by \cite[8.2]{vandenBergh_Existence_for_dualizing}, if a rigid dualizing complex exists, it is unique up to isomorphism. 
\end{rem}

\vskip5pt

\begin{examples}
  It is straightforward to verify that $k$ is a dualizing module for $A=k$. Dualizing complexes for schemes were introduced in \cite[V]{Hartshorne_RD}. As is shown there, $\Z$ is a dualizing module for $\Z$. Moreover, if $A$ is a commutative local Noetherian Gorenstein ring, then $A$ itself is a dualizing module for $A$ \cite[3.3.7]{BrunsHerzog_CMrings}. Any commutative $k$-algebra of essentially finite type has a dualizing complex \cite[47.15.11]{stacks-project}. Examples of non-commutative algebras possessing a dualizing complex can be found in \cite{Yekutieli_Dualizing_graded,YekutieliZhang,Yekutieli_Dualizing_Ug,vandenBergh_Existence_for_dualizing}.
\end{examples}

Given a dualizing complex $R$ as in \ref{dfn dualizing complex} we consider the functors
\[\begin{array}{rl}
\bbD := \RHom_A(-,R):  & D(A\hmod) \lra D(B^\o\hmod) \;,\\
&\\
\bbD^\o := \RHom_{B^\o}(-,R):  & D(B^\o\hmod) \lra D(A\hmod) \;.
\end{array}\]

\vskip5pt

Let $A$ be a left noetherian $k$-algebra. Denote by $D_f(A\hmod)$ and $D^b_f(A\hmod)$ the triangulated subcategories of complexes whose cohomology modules are finitely generated. 

\vskip5pt

\begin{prop} {\rm \cite[1.3]{YekutieliZhang}} Let $A,B$ be as in \ref{dfn dualizing complex} and let $R$ in $D(A \ot_k B^\ot\hmod)$ be a dualizing complex.

\vskip5pt

(i) For any $M$ in $D_f(A\hmod)$ one has $\bbD(M) \in D_f(B^\o\hmod)$ and $M \cong \bbD^\o(\bbD(M))$. 

\vskip5pt

(ii) The functors $\bbD$ and $\bbD^\o$ determine a duality, i.e., an anti-equivalence, of triangulated categories between $D_f(A\hmod)$ and $D_f(B^\o\hmod)$, restricting to a duality between $D^b_f(A\hmod)$ and $D^b_f(B^\o\hmod)$.
\end{prop}

\vskip5pt

We now turn to the concrete case of the universal enveloping algebra.

\begin{thm}\label{dualizing module for Ug} {\rm \cite[Thm. A]{Yekutieli_Dualizing_Ug}} Let $\frh$ be a finite-dimensional Lie algebra $\frh$ over $E$. Then 
\[\Big(\Uh \ot_E \bigwedge^{\dim_E(\frh)} \frh\Big) [\dim_E(\frh)]\]

\vskip5pt

is a rigid dualizing complex for $\Uh$. Here we consider $\bigwedge^{\dim_E(\frh)} \frh$ as a left $\frh$-module with the trivial action and as a right $\frh$-module with the adjoint action. 
\end{thm}

\vskip5pt

\begin{rem}  When we apply \ref{dualizing module for Ug} to the reductive Lie algebra $\frg_E$, we find that $R = \UgE [\dim_F(\frg)]$ is a dualizing module, because $\bigwedge^{\dim_F(\frg)} \frg_E$ is the trivial one-dimensional representation. If $W$ is a finite-dimensional $\frg_E$-module, it follows from \ref{ExtofInd} that
\[\RHom_U(W,U[\dim_F(\frg)]) \; \cong \; W' \;,\]

\vskip5pt

i.e., this duality functor maps finite-dimensional representations to  finite-dimensional representations, considered as complexes concentrated in degree zero. While this is a nice property of this rigid dualizing complex, we in fact rather work with $R = \UgE$ itself in the rest of this dissertation.
\end{rem}

\subsection{The duality functors on \texorpdfstring{$\DbcOi$}{} and \texorpdfstring{$\DbcOia$}{}}

\begin{prop}
(i) $\RHom_U(-,U)$ preserves the subcategory $D^b_{\cOi}(\UgE)$ and thus induces a (contravariant) functor
\begin{numequation}\label{duality1} \RHom_U(-,U): D^b_{\cOi}(\UgE) \lra D^b_{\cOi}(\UgE)
\end{numequation}
which is an anti-equivalence and an involution. 

\vskip5pt

(ii) $\RHom_U(-,U)$ preserves the subcategory $D^b_{\cOia}(\UgE)$ and thus induces a (contravariant) functor
\begin{numequation}\label{duality2} \RHom_U(-,U): D^b_{\cOia}(\UgE) \lra D^b_{\cOia}(\UgE)
\end{numequation}
which is an anti-equivalence and an involution. 
\end{prop} 

\begin{proof} This follows from \ref{Oinftystability} and \ref{dualizing module for Ug}, together with \ref{shift remark}.
\end{proof}

\begin{para} {\it Motivation.}  In another paper, we will consider an exact functor $\vcF = \vcFGB: \cOia \ra \DGmod$, where $D(G)$ is the locally analytic distribution algebra of $p$-adic reductive group $G$ with Lie algebra $\frg$. As this functor is exact, it gives rise to a functor $\vcF: \DbcOia \ra D^b(\DGmod)$. Our aim is to understand how this functor behaves when pre-composed with the duality functor in \ref{duality2}. Yet in order to make sense of this question, the duality functor must first be defined on $\DbcOia$. This is indeed possible by the result \ref{extensionfull} below. We start with discussing some relevant concepts from \cite[sec. 2.2]{CoulembierMazorchuk_III}.
\end{para}

\vskip5pt

\begin{para}\label{dfn ext full} {\it Extension full subcategories.}  Let $\cB$ be a full abelian subcategory of an abelian category $\cA$, and we assume that the inclusion functor $\iota: \cB \ra \cA$ is exact. Because of this assumption, $\iota$ induces homomorphisms of extension groups 
\begin{numequation}\label{mapsonExts}
\Ext^i_\cB(M,N) \ra \Ext^i_\cA(M,N) \;,
\end{numequation}
for any two objects $M,N$ of $\cB$ and $i \in \Z_{\ge 0}$. The extension groups are here Yoneda Ext groups. These maps are in general neither injective nor surjective. $\cB$ is called {\it extension full} in $\cA$ if and only if \ref{mapsonExts} is an isomorphism for all objects $M,N$ of $\cB$ and $i \in \Z_{\ge 0}$. When $i=0$ this map is always bijective, as $\cB$ was assumed to be full. If $\cB$ is a Serre subcategory (\ref{Serre}), then the maps \ref{mapsonExts} for $i=1$ are bijective. We continue with some useful results from \cite{CoulembierMazorchuk_Extfullness_GelfandZeitlin,CoulembierMazorchuk_III}.
\end{para}

\vskip5pt

\begin{prop}\label{ext full criterion} {\rm \cite[Prop. 5 in sec. 2.3]{CoulembierMazorchuk_Extfullness_GelfandZeitlin}} Consider an associative algebra $A$, a full abelian subcategory $\cA$ in $A\hmod$,\footnote{Note that we denote by $A\hmod$ the category of {\it all} $A$-left-modules, whereas this category is denoted by $A\mbox{-}{\rm Mod}$ in \cite{CoulembierMazorchuk_Extfullness_GelfandZeitlin}.} and a full abelian subcategory $\cB$ of $\cA$. Assume that these data satisfy the following conditions: 

\begin{enumerate}
    \item[(i)] $\cB$ is a Serre subcategory of $A\hmod$. 
    \item[(ii)] For every surjective morphism $\alpha: M \ra N$, with $M$ in $\cA$ and $N$ in $\cB$, there is a $Q$ in $\cB$ and an injective morphism $\beta: Q \ra M$ such that the composition $\alpha \circ \beta: Q \ra N$ is surjective. 
\end{enumerate}

\vskip5pt

Then $\cB$ is extension full in $\cA$.
\end{prop}

\vskip5pt

\begin{prop}\label{extfull implies full on derived} {\rm \cite[sec. 3, Prop. 8]{CoulembierMazorchuk_III}} Let $\iota: \cB \ra \cA$ be the inclusion of an abelian full subcategory $\cB$ of an abelian category $\cA$. We assume that $\iota$ is exact. Then $\cB$ is extension full in $\cA$ if and only if the functor 
\[D^b(\iota): D^b(\cB) \lra D^b(\cA)\footnote{Note that there is a typo in the statement of  \cite[Prop. 8]{CoulembierMazorchuk_III}, where the target category is denoted by $\cC^b(\cA)$ but $D^b(\cA)$ is meant there.}\]

\vskip5pt

induced by $\iota$ is fully faithful and triangulated (i.e., sends distinguished triangles to distinguished triangles).  
\end{prop}

\vskip5pt

\begin{thm}\label{extensionfull} The categories $\cOi$ and $\cOia$ are both extension full in $\Ugmod$. The canonical functors 
\begin{numequation}\label{CoulembierMazorchuk} \DbcOi \lra D^b_{\cOi}(\UgE)
\end{numequation}
\begin{numequation}\label{CoulembierMazorchuk algebraic} \DbcOia \lra D^b_{\cOia}(\UgE)
\end{numequation}
are equivalences of categories.
\end{thm}

\begin{proof} It is an easy consequence of \ref{ext full criterion} that $\Ugmodfg$, the category of finitely generated $\UgE$-modules, is extension full in $\Ugmod$ \cite[Cor. 6 (ii) in sec. 2.3]{CoulembierMazorchuk_Extfullness_GelfandZeitlin}. And by 
\cite[Thm. 16 in sec. 5]{CoulembierMazorchuk_III} the category $\cOi$ is extension full in $\Ugmodfg$ (note that in \cite{CoulembierMazorchuk_III} the category $\frg\mbox{-}{\rm mod}$ is the category of finitely generated $\Ug$-modules). Hence $\cOi$ is extension full in $\Ugmod$.

\vskip5pt

Now we show that $\cOia$ is extension full in $\cOi$. We have already seen that $\cOia$ is a Serre subcategory of $\Ugmod$ \ref{key facts Oinfty}. Let $M \ra N$ be a surjection in $\cOi$ with $N$ an object in $\cOia$. We want to show that condition (ii) in \ref{ext full criterion} holds. It follows from the discussion in \ref{z-blocks} that we may assume that $N$ is contained in a $\frz$-block $\cOi_{|\lambda|}$ where $\lambda$ is necessarily algebraic. Then $\pr_{|\lambda|}(M) \ra \pr_{|\lambda|}(N) = N$ is still surjective, and $\pr_{|\lambda|}(M)$ is in $\cOia$. By \ref{ext full criterion} we conclude that $\cOia$ is extension full in $\cOi$, hence it is extension full in $\Ugmod$.

\vskip5pt

By \ref{extfull implies full on derived} the canonical functors 
\[\DbcOi \lra D^b(\UgE) \;\;\; \mbox{ and } \;\;\; \DbcOia \lra D^b(\UgE)\]

\vskip5pt

are fully faithful. These functors factor through the full subcategories $D^b_{\cOi}(\UgE)$ and\linebreak $D^b_{\cOia}(\UgE)$, respectively. We thus have to show that the functors 
\begin{numequation}\label{can functors}
\DbcOi \lra D^b_{\cOi}(\UgE) \;\;\; \mbox{ and } \;\;\; \DbcOia \lra D^b_{\cOia}(\UgE)
\end{numequation}
are essentially surjective. To see this, we argue as Bernstein and Lunts in \cite[1.9.5, 1.9.4]{BernsteinLunts_Localization}. Given a complex $M$ in $D^b_{\cOi}(\UgE)$ we define $ca(M)$, the {\it cohomological amplitude}\footnote{We have not found in the literature a commonly accepted definition and use this only as an ad hoc definition for the purpose of this proof.} $M$, to be zero if $H^{\bdot}(M) = 0$ and to be $s-i+1$ if $M$ has non-vanishing cohomology, where $s = \max\{i \in \Z \midc H^i(M) \neq 0\}$ and $i = \min\{i \in \Z \midc H^i(M) \neq 0\}$. If $ca(M) \le 1$ then $M$ is obviously isomorphic to an object in the image of this functor. Let $n \ge 1$ and suppose that any object $N$ with $ca(N) \le n$ is in the image of this functor.  Assume $ca(M) = n+1$ and choose $i<j$ such that $H^i(M) \neq 0 \neq H^j(M)$. Using \cite[7.3.10]{Yekutieli_DerivedCategories}, we have objects $\smt^{\le i}(M)$ and $\smt^{\ge i+1}(M)$ of $D^b(\UgE)$, the smart truncations of $M$, together with morphisms $e: \smt^{\le i}(M) \ra M$ and $p: M \ra \smt^{\ge i+1}(M)$, such that 
\[\forall k>i: \; H^k(\smt^{\le i}(M)) = 0 \;\;\; \mbox{ and } \;\;\; \forall k \le i: \; H^k(e): H^k(\smt^{\le i}(M)) \isom H^k(M)\]

and 
\[\begin{array}{l}\forall k < i+1: \; H^k(\smt^{\ge i+1}(M)) = 0 \;,\\
\\
\forall k \ge i+1: \; H^k(p): H^k(\smt^{\ge i+1}(M)) \isom H^k(M) \;,
\end{array}\]

\vskip5pt

which shows that these truncations belong to $D^b_{\cOi}(\UgE)$ (resp. $D^b_{\cOia}(\UgE)$) if this is true for $M$. Furthermore,
there is a distinguished triangle in $D^b(\UgE)$
\[\smt^{\le i}(M) \xrightarrow{e} M \xrightarrow{p} \smt^{\ge i+1}(M) \xrightarrow{\theta} \smt^{\le i}(M)[1]\] 

\vskip5pt
If we `turn' this distinguished triangle twice we get the distinguished triangle
\[\smt^{\ge i+1}(M) \xrightarrow{\theta} \smt^{\le i}(M)[1] \xrightarrow{-e[1]} M[1] \xrightarrow{-p[1]} \smt^{\ge i+1}(M)[1]\]

\vskip5pt

Because the functors \ref{can functors} are fully faithful, the essential images of these functors are strictly full triangulated subcategories. Since the objects $\smt^{\le i}(M)$ and $\smt^{\ge i+1}(M)[1]$ belong to this subcategory (by induction, as their cohomological amplitide is at most $n$), and since $\theta$ comes from a morphism in $\DbcOi$ (resp. $\DbcOia$), the latter distinguished triangle belongs to this essential image, and this means that $M$ is in the essential image. 
\end{proof}

\vskip5pt

\begin{para} {\it The duality functors on $\DbcOi$ and $\DbcOia$.}  We now define the functors 
\begin{numequation}
\bbD^\frg: \DbcOi \lra \DbcOi \hskip20pt \mbox{ and } \hskip20pt \bbD^\frg_\alg: \DbcOia \lra \DbcOia
\end{numequation}\label{duality functors}
by composing the functors in \ref{duality1} (resp. \ref{duality2}) with the equivalences in \ref{CoulembierMazorchuk} (resp. \ref{CoulembierMazorchuk algebraic}):
\[\bbD^\frg: \; \DbcOi \isom D^b_{\cOi}(\UgE) \xrightarrow{\RHom_U(-,U)} D^b_{\cOi}(\UgE) \isom \DbcOi\]

\vskip5pt

and
\[\bbD^\frg_\alg: \; \DbcOia \isom D^b_{\cOia}(\UgE) \xrightarrow{\RHom_U(-,U)} D^b_{\cOia}(\UgE) \isom \DbcOia \;.\]
\end{para}

\bibliographystyle{alpha}
\bibliography{mybib}

\end{document}